\documentclass[a4paper, 12pt, twoside, reqno]{amsart}
\usepackage{calrsfs}
\usepackage{amsmath}
\usepackage{amssymb}
\usepackage{amsthm}
\usepackage{amscd}
\usepackage{amsopn}
\usepackage{eucal}
\usepackage[UKenglish]{babel}

\usepackage{etex}
\usepackage{pictex}
\usepackage[all,cmtip]{xy}

\newtheoremstyle{Teorema}{5pt}{5pt}{\it}{}{\bf}{.}{ }{}
\theoremstyle{Teorema}
\newtheorem{Theorem}{Theorem}[section]
\newtheorem*{MainTheorem}{Theorem A}
\newtheorem*{ThrmB}{Theorem B}
\newtheorem*{Thrm}{Theorem}

\newtheorem{Corollary}[Theorem]{Corollary}
\newtheorem{Proposition}[Theorem]{Proposition}
\newtheorem{Definition}[Theorem]{Definition}
\newtheorem{Lemma}[Theorem]{Lemma}
\newtheoremstyle{Annotazione}{5pt}{5pt}{\rm}{}{\bf}{.}{ }{}
\theoremstyle{Annotazione}
\newtheorem{Remark}[Theorem]{Remark}
\newtheorem*{Ackno}{Acknowledgement}
\newtheorem*{Outline}{Outline}

\def\SL{\operatorname{SL}}
\def\PSL{\operatorname{PSL}}
\def\PGL{\operatorname{PGL}}
\def\GL{\operatorname{GL}}
\def\bba{\mathbb{A}}
\def\bbz{\mathbb{Z}}
\def\bbr{\mathbb{R}}
\def\bbq{\mathbb{Q}}
\def\bbc{\mathbb{C}}
\def\bbh{\mathbb{H}}

\def\bbp{\mathbb{P}}
\def\bbf{\mathbb{F}}
\def\bbn{\mathbb{N}}

\def\PO{\mathrm{P}\mathrsfs{O}^1}
\def\tr{\operatorname{tr}}

\def\Gal{\operatorname{Gal}}

\begin{document}
\selectlanguage{UKenglish}
\title{Modular Embeddings and Rigidity for Fuchsian Groups}
\author{Robert A. Kucharczyk}
\address{Universit\"{a}t Bonn\\
Mathematisches Institut\\
Endenicher Allee 60\\
53115 Bonn\\
Germany}
\email{rak@math.uni-bonn.de}
\thanks{Research supported by the European Research Council}

\begin{abstract}
We prove a rigidity theorem for semi-arithmetic Fuchsian groups: If $\Gamma_1$, $\Gamma_2$ are two semi-arithmetic lattices in $\PSL (2,\bbr )$ virtually admitting modular embeddings and $f\colon\Gamma_1\to\Gamma_2$ is a group isomorphism that respects the notion of congruence subgroups, then $f$ is induced by an inner automorphism of $\PGL (2,\bbr )$.\end{abstract}
\maketitle
\tableofcontents

\section{Introduction}

\noindent In 1968 George Mostow published his famous Rigidity Theorem \cite{MR0236383}: if $M_1$ and $M_2$ are two closed oriented hyperbolic manifolds of dimension $n\ge 3$ and $f\colon\pi_1(M_1)\to\pi_1(M_2)$ is a group isomorphism, then there exists a unique isometry $M_1\to M_2$ inducing $f$. This can be reformulated as a statement about lattices in the orientation-preserving isometry groups $\operatorname{PSO}(1,n)$ of hyperbolic $n$-space $\mathbf{H}^n$:
\begin{Thrm}[Mostow] Let $n\ge 3$ and let $\Gamma_1,\Gamma_2\subset\operatorname{PSO}(1,n)$ be cocompact lattices. Let $f\colon\Gamma_1\to\Gamma_2$ be an isomorphism of abstract groups. Then $f$ is conjugation by some element of the full isometry group $\operatorname{PO}(1,n)$ of $\mathbf{H}^n$, in particular $f$ extends to an algebraic automorphism of $\operatorname{PSO}(1,n)$.
\end{Thrm}
This has later been generalised by various authors; in particular, the condition that $\Gamma_j$ be cocompact can be weakened to having finite covolume, see \cite{MR0385005}. The condition that $n\neq 2$, however, is necessary: two-dimensional hyperbolic manifolds are the same as hyperbolic Riemann surfaces, which are well-known to admit deformations.

As a model for the hyperbolic plane take the upper half-plane $\mathfrak{H}=\{\tau\in\bbc\mid\operatorname{Im}\tau >0\}$, so its orientation-preserving isometry group becomes identified with $\PSL (2,\bbr )$ via M\"{o}bius transformations. In this article we prove that a variant of Mostow Rigidity does hold in $\operatorname{Isom}^+(\mathfrak{H})=\PSL (2,\bbr )$ if we restrict ourselves to a certain class of lattices, for which congruence subgroups are defined, and demand that the group isomorphism preserves congruence subgroups.

We first state our result in the simpler case of arithmetic groups. Recall that given a totally real number field $k\subset\bbr$, a quaternion algebra $B$ over $k$ which is split over the identity embedding $k\to\bbr$ and ramified over all other infinite places of $k$, an order $\mathrsfs{O}\subset B$ and an isomorphism $\varphi\colon B\otimes_k\bbr\to\mathrm{M}(2,\bbr )$ we obtain a group homomorphism $\varphi\colon\mathrsfs{O}^1\to\PSL (2,\bbr )$ whose image is a lattice, where $\mathrsfs{O}^1$ is the group of units in $\mathrsfs{O}$ with reduced norm one. A lattice $\Gamma\subset\PSL (2,\bbr )$ is called \emph{arithmetic} if $\Gamma$ is commensurable to some such $\varphi (\mathrsfs{O}^1)$.

For a nonzero ideal $\mathfrak{n}\subset\mathfrak{o}_k$ we then define the \emph{principal congruence subgroup}
$$\mathrsfs{O}^1(\mathfrak{n})=\{ b\in\mathrsfs{O}^1\mid b-1\in\mathfrak{n}\cdot\mathrsfs{O}\} .$$
If $\Gamma$ contains a subgroup of finite index in $\varphi (\mathrsfs{O}^1)$ we set $\Gamma (\mathfrak{n})=\Gamma \cap\varphi (\mathrsfs{O}^1(\mathfrak{n}))$, and a subgroup of $\Gamma$ is a \emph{congruence subgroup} if it contains some $\Gamma (\mathfrak{n})$.

\begin{Thrm}[special case of Theorem A below]
Let $\Gamma_1,\Gamma_2\subset\PSL (2,\bbr )$ be arithmetic Fuchsian groups, and let $f\colon\Gamma_1\to\Gamma_2$ be an isomorphism of abstract groups such that for every subgroup $\Delta\subseteq\Gamma_1$ of finite index, $\Delta$ is a congruence subgroup of $\Gamma_1$ if and only if $f(\Delta )$ is a congruence subgroup of $\Gamma_2$.

Then there exists $a\in\PGL (2,\bbr )$ such that $f$ is conjugation by $a$. In particular, $\Gamma_2=a\Gamma_1a^{-1}$.
\end{Thrm}

Now both the notion of congruence subgroup and our result can be extended to a larger class of Fuchsian groups.

For a subgroup $\Gamma\subseteq\PSL (2,\bbr )$ denote the preimage in $\SL (2,\bbr )$ by $\tilde{\Gamma }$. A lattice $\Gamma\subset\PSL (2,\bbr )$ is called \emph{semi-arithmetic} if $\tr^2\gamma$ is a totally real algebraic integer for each $\gamma\in\tilde{\Gamma }$; this notion is invariant under commensurability. It was introduced in \cite{MR1745404}, and many classes of Fuchsian groups are semi-arithmetic:
\begin{enumerate}
\item Arithmetic lattices are semi-arithmetic.
\item All Fuchsian triangle groups $\Delta (p,q,r)$ are semi-arithmetic. However, they fall into infinitely many commensurability classes, only finitely many of which are arithmetic, see \cite{MR0429744}.
\item In \cite{MR1745404} further examples of semi-arithmetic groups which are not arithmetic were constructed by giving explicit generators.
\item The theory of flat surfaces provides for another construction of semi-arithmetic groups. If $X$ is a closed Riemann surface and $\omega$ is a holomorphic one-form on $X$ which is not identically zero, a simple geometric construction yields the Veech group\footnote{The name first appeared in \cite{MR1397686} but these groups were studied before from different points of view, see \cite{MR1005006}.} $\SL (X,\omega )$ which is a discrete subgroup of $\SL (2,\bbr )$. In certain cases the Veech group is a lattice, and then its image in $\PSL (2,\bbr )$ is a semi-arithmetic group by \cite[Theorems 5.1, 5.2]{MR1992827} and \cite[Proposition 2.6]{MR2188128}. Veech groups are never cocompact, see \cite[p. 509]{MR2186246}, therefore a Veech group which is a lattice is arithmetic if and only if it is commensurable to $\SL (2,\bbz )$.\footnote{For a complete characterisation of $(X,\omega )$ whose Veech group is arithmetic see \cite[Theorem 4]{MR1397686}.} In \cite{MR1992827} we find, for every real quadratic number field $k$, the construction of a lattice Veech group contained in $\SL (2,\mathfrak{o}_k)$ which is therefore semi-arithmetic but not arithmetic.
\end{enumerate}
Examples (ii) and (iv) intersect: in \cite[Theorem 6.12]{MR2680418} it is proved that all non-cocompact triangle groups $\Delta (p,q,\infty )$ are commensurable to some Veech group. On the other hand, cocompact triangle groups can never be Veech groups, and only finitely many of the examples in \cite{MR1992827} are commensurable with triangle groups.

The generalisation of the notion of congruence subgroups to semi-arithmetic groups is a bit involved; we refer the reader to section \ref{SectionSemiArithmetic}.

Now the conclusion of Theorem A does not hold for general semi-arithmetic groups; we need to impose one more condition which is the existence of a \emph{modular embedding}: let $\Gamma\subset\PSL (2,\bbr )$ be a semi-arithmetic subgroup and let $k$ be the number field generated by all $\tr^2\gamma$ with $\gamma\in\tilde{\Gamma }$. Then for every embedding $\sigma\colon k\to\bbr$ there exists a group embedding $i_{\sigma }\colon\tilde{\Gamma }\to\SL (2,\bbr )$, unique up to conjugation in $\GL (2,\bbr )$, such that $\tr^2i_{\sigma }(\gamma )=\sigma (\tr^2\gamma )$ for every $\gamma\in\tilde{\Gamma }$, see \cite[Remark 4]{MR1745404}. The original group $\Gamma$ is arithmetic precisely if no $i_{\sigma }(\tilde{\Gamma })$ for $\sigma$ different from the identity embedding contains a hyperbolic element. In general, let $\sigma_1,\ldots ,\sigma_r$ be those embeddings $\sigma$ for which $i_{\sigma }(\tilde{\Gamma })$ contains a hyperbolic element. Then the coordinate-wise embedding $(i_{\sigma_1},\ldots ,i_{\sigma_r})\colon\Gamma\to\PSL (2,\bbr )^r$ maps $\Gamma$ to an irreducible arithmetic group $\Lambda\subset\PSL (2,\bbr )^r$; for the precise construction see section \ref{SectionModularEmbedding}.

We note that if $\Gamma$ is not already arithmetic itself, it is mapped into $\Lambda$ with Zariski-dense image of infinite index; such groups are called \emph{thin}. This is essentially due to S.\ Geninska \cite[Proposition 2.1 and Corollary 2.2]{MR2968231}; we explain it below in Corollary \ref{SemiArithmeticAreThin}.

Now $\Lambda$ acts on $\mathfrak{H}^r$ by coordinate-wise M\"{o}bius transformations, and a \emph{modular embedding} for $\Gamma$ is then a holomorphic map $F\colon\mathfrak{H}\to\mathfrak{H}^r$ equivariant for $\Gamma\to\Lambda$.
\begin{enumerate}
\item If $\Gamma$ is arithmetic, then $r=1$ and $\Lambda$ contains $\Gamma$ as a finite index subgroup. We may take $F(\tau )=\tau$ as a modular embedding.
\item All Fuchsian triangle groups admit modular embeddings, see \cite[Theorem p. 96]{MR1075639}.
\item Most of the new examples of semi-arithmetic groups in \cite{MR1745404} do not admit modular embeddings, see \cite[Corollary 4]{MR1745404}.
\item Veech groups which are lattices always admit modular embeddings, see \cite[Corollary 2.11]{MR2188128}. This solves \cite[Problem 1]{MR1745404} which asks whether every Fuchsian group admitting a modular embedding is arithmetic or commensurable with a triangle group: there exist Veech groups which are neither\footnote{Almost all of McMullen's genus two examples in \cite{MR1992827} do the job: only finitely many real quadratic fields appear as invariant trace fields of triangle groups, so if $k$ is not among them, then any lattice Veech group with trace field $k$ cannot be commensurable to a triangle group, and it cannot be arithmetic either since it is not cocompact.}, but do admit modular embeddings.
\end{enumerate}

More generally, we say $\Gamma$ \emph{virtually admits a modular embedding} if some finite index subgroup of $\Gamma$ admits one.
\begin{MainTheorem}
For $j=1,2$, let $\Gamma_j\subset\PSL (2,\bbr )$ be semi-arithmetic lattices which virtually admit modular embeddings. Let $f\colon\Gamma_1\to\Gamma_2$ be an isomorphism of abstract groups such that for every subgroup $\Delta\subseteq\Gamma_1$ of finite index, $\Delta$ is a congruence subgroup of $\Gamma_1$ if and only if $f(\Delta )$ is a congruence subgroup of $\Gamma_2$.

Then there exists $a\in\PGL (2,\bbr )$ such that $f$ is conjugation by $a$. In particular, $\Gamma_2=a\Gamma_1a^{-1}$.
\end{MainTheorem}
This theorem will be proved in section \ref{SectionCongruenceRigidity}. It rests on the following result on congruence subgroups in semi-arithmetic groups, which may be of independent interest.
\begin{ThrmB}
Let $\Gamma\subset\PSL (2,\bbr )$ be a semi-arithmetic lattice satisfying the trace field condition\footnote{This is a technical condition which is always satisfied after passing to a finite index subgroup, see Definition \ref{DefTFC}.}, with trace field $k$. Then there exists a finite set $S(\Gamma )$ of rational primes with the following property:
\begin{enumerate}
\item If $\mathfrak{p}$ is a prime ideal in $k$ not dividing any element of $S(\Gamma )$, then $\Gamma /\Gamma (\mathfrak{p})\simeq\PSL (2,\mathfrak{o}_k/\mathfrak{p})$.
\item If $q$ is a rational prime power not divisible by any element of $S(\Gamma )$ and $\Delta$ is a normal congruence subgroup of $\Gamma$ with $\Gamma /\Delta \simeq\PSL (2,q)$, then there exists a unique prime ideal $\mathfrak{p}$ of $k$ of norm $q$ with $\Delta =\Gamma (\mathfrak{p})$.
\end{enumerate}
\end{ThrmB}
Here, (i) is a combination of Proposition \ref{CorollaryGammaLocallyCongruence} and Lemma \ref{OoneModOonepIsSLTwo}; (ii) is Proposition \ref{CongruenceQuotientsGivePrimes}.

In particular, the information which groups $\PSL (2,q)$ appear how often as congruence quotients determines the splitting behaviour of all but finitely many primes in $k$ (see Remark \ref{ReconstructionOfArithmeticSimilarityClass}). On the other hand, allowing noncongruence quotients we get many more finite groups. The collection of all these finite groups will determine the abstract isomorphism type of a Fuchsian lattice, but of course no more, see \cite[Theorem 1.1]{arxiv14013645}.
\begin{Outline}
In sections 2 and 3 we fix notations and recall standard results on the group $\PSL (2)$, both over the reals and over finite fields. In sections 4 and 5 we introduce semi-arithmetic subgroups of $\PSL (2,\bbr )$ and study their congruence subgroups. The object of section 6 is the deduction of a statement about $\PSL (2)$ from an analogous result for $\SL (2)$ by Culler and Shalen \cite[Proposition 1.5.2]{MR683804}: a finitely generated subgroup of $\PSL (2,\bbr )$ is determined up to conjugacy by its squared traces. This allows us to work with numbers instead of matrices in the remainder of the article. In section 7 we formally define modular embeddings and discuss some consequences of their existence. Then in section 8 the previous observations are used to prove Theorem A and the hard part of Theorem B. Section 9 presents an example with two arithmetic groups, sharpening the statement of Theorem A considerably in this special case. Finally section 10 discusses some possible and impossible generalisations.
\end{Outline}
\begin{Ackno}
This work is part of the author's ongoing dissertation project at the University of Bonn. He wishes to thank his supervisor, Ursula Hamenst\"{a}dt, for many fruitful discussions, and helpful comments on earlier versions of this article.
\end{Ackno}

\section{Traces on $\PSL (2)$ and M\"{o}bius Transformations}\label{SectionTraces}

\noindent For every ring $A$ we set $\PGL (2,A) = \GL (2,A) /A^{\times }$ where $A^{\times }$ is embedded by means of scalar matrices. We also set $\PSL (2,A)=\SL (2,A)/\{\pm\mathbf{1}\}$. There is an obvious homomorphism $\PSL (2,A)\to \PGL (2,A)$, but in general it is neither injective nor surjective.

Let $k$ be a field. The determinant homomorphism $\GL (2,k)\to k^{\times }$ descends to a homomorphism $\PGL (2,k)\to k^{\times }/(k^{\times })^2$, and we obtain a short exact sequence
\begin{equation}\label{SesPSL}
1\longrightarrow \PSL (2,k)\longrightarrow \PGL (2,k)\longrightarrow k^{\times }/(k^{\times })^2\longrightarrow 1.
\end{equation}
In particular, $\PSL (2,\bbc )$ and $\PGL (2,\bbc )$ are naturally isomorphic whereas for $k=\bbr$ or a finite field of odd characteristic, $\PSL (2,k)$ becomes identified with an index two normal subgroup of $\PGL (2,k)$.

Note that since $\PSL (2,k)$ is a normal subgroup of $\PGL (2,k)$, the latter operates faithfully on the former by conjugation. Since $\tr (-g)=-\tr g$, the squared trace map $\tr^2\colon\SL (2,k)\to k$ descends to a map
$$\tr^2\colon\PSL (2,k)\to k,\quad\{ g,-g\}\mapsto (\tr g)^2.$$
For $k=\bbr$ we also define
$$\lvert\tr\rvert\colon\PSL (2,\bbr )\to\bbr,\quad \{ g,-g\}\mapsto\lvert\tr g\rvert .$$
Let $\mathfrak{H}=\{\tau\in\bbc\mid\operatorname{Im} (\tau )>0\}$ be the upper half-plane. The group $\SL (2,\bbr )$ operates on $\mathfrak{H}$ in the well-known way by M\"{o}bius transformations, descending to a faithful action by $\PSL (2,\bbr )$. This in fact identifies $\PSL (2,\bbr )$ with both the group of holomorphic automorphisms and that of orientation-preserving isometries (for the Poincar\'{e} metric) of $\mathfrak{H}$. Elements of $\PSL (2,\bbr )$ can be categorised by their behaviour on $\mathfrak{H}$, see \cite[section 1.3]{MR1177168}:
\begin{Proposition}
Let $\pm\mathbf{1}\neq g\in\PSL (2,\bbr )$. Then $g$ belongs to exactly one of the following classes:
\begin{enumerate}
\item $g$ is \emph{elliptic:} it has a unique fixed point in $\mathfrak{H}$, and $\tr^2 g<4$.
\item $g$ is \emph{parabolic:} it has a unique fixed point in $\bbp^1(\bbr )$, but not in $\mathfrak{H}$. Its squared trace satisfies $\tr^2 g=4$.
\item $g$ is \emph{hyperbolic:} it has two distinct fixed points in $\bbp^1(\bbr )$, one of them repelling and one of them attracting, but no fixed points in $\mathfrak{H}$. Its squared trace satisfies $\tr^2 g>4$.
\end{enumerate}
\end{Proposition}

\section{The Finite Groups $\PSL (2,q)$}\label{SectionFiniteGroups}

\noindent Next we study $\PSL (2)$ over finite fields. With $\bbf_q$ being the field of $q$ elements we also write $\PSL (2,q)$ instead of $\PSL (2,\bbf_q)$.

\begin{Proposition}
If $q>3$ is an odd prime power, $\PSL (2,q)$ is a simple group of order $\frac{1}{2}q(q^2-1)$. Furthermore $\PSL (2,q)\simeq\PSL (2,q')$ if and only if $q=q'$.
\end{Proposition}
\begin{proof}
The simplicity of $\PSL (2,q)$ is a well-known fact, see e.g.\ \cite[section 3.3.2]{MR2562037}. The order of $\PSL (2,q)$ is easily calculated using (\ref{SesPSL}), for instance. The function $q\mapsto \frac{1}{2}q(q^2-1)$ is strictly increasing on $\bbn$, therefore if $\PSL (2,q)$ and $\PSL (2,q')$ have the same orders, $q=q'$.
\end{proof}
As remarked in section \ref{SectionTraces}, $\PGL (2,q)$ operates by conjugation on $\PSL (2,q)$. Furthermore the Frobenius automorphism $\varphi\colon\bbf_q\to\bbf_q$ defined by $\varphi (x)=x^p$, where $p$ is the prime of which $q$ is a power, defines an automorphism $\varphi$ of $\PSL (2,q)$. The following is also well-known, see e.g.\ \cite[Theorem 3.2.(ii)]{MR2562037}:
\begin{Proposition}
The automorphism group of $\PSL (2,q)$ is generated by $\PGL (2,q)$ and $\varphi$.
\end{Proposition}
In particular if $q=p$ is a prime, then every automorphism of $\PSL (2,p)$ is the restriction of an inner automorphism of $\PGL (2,p)$, and the map $\tr^2\colon\PSL (2,p)\to\bbf_p$ is invariant under all automorphisms. So the following definition works:
\begin{Definition}\label{AbstractTraces}
Let $G$ be a finite group which is abstractly isomorphic to some $\PSL (2,p)$ for an odd prime $p$. Then the map $\tr_G^2\colon G\to\bbf_p$
is defined as follows: choose some automorphism $\alpha\colon G\to\PSL (2,p)$, then set $\tr_G^2=\tr^2\circ\alpha$.
\end{Definition}
If $p$ is replaced by a prime power $q$, the corresponding map on $G$ is only are well-defined up to automorphisms of $\bbf_q$, i.e.\ we may define a map $\tr_G^2\colon G\to\bbf_q/\operatorname{Aut}\bbf_q$.
\begin{Lemma}\label{LemmaOnFiniteGroups}
Let $n\in\bbn$ and let $q_1,\ldots ,q_n,q'$ be odd prime powers. Let
$$\beta\colon G=\PSL (2,q_1)\times\cdots\times\PSL (2,q_n)\to\PSL (2,q')$$
be a group epimorphism. Then there is a $1\le j\le n$ such that $q'=q_j$ and for some automorphism $\alpha$ of $\PSL (2,q')$ we can write $\beta =\alpha\circ\operatorname{pr}_j$, where $\operatorname{pr}_j$ is the projection on the $j$-th factor.
\end{Lemma}
\begin{proof}
By the Jordan--H\"{o}lder theorem, the only simple quotients of $G$ are the $\PSL (2,q_j)$, so $q'=q_j$ for some $j$.

We now proceed by induction on $n$. For $n=1$ the lemma is trivial, so assume the lemma has been proved for $n$. Let $\beta\colon G\to \PSL (2,q')$ be an epimorphism where $G$ has $n+1$ factors. For cardinality reasons it cannot be injective, so there exists some $g\in G\smallsetminus\{ 1\}$ with $\beta (g)=1$. Write $g=(g_1,\ldots ,g_{n+1})$, then $g_j\neq 1$ for some $j$; for simplicity of notation assume that $j=n+1$. Since $\PSL (2,q_{n+1})$ has trivial centre, there exists some $h_{n+1}\in G$ which does not commute with $g_{n+1}$. Then set
$$h=(1,\ldots ,1,h_{n+1})\in G$$
and compute
$$1=\beta(h)\beta(h^{-1})=\beta (ghg^{-1}h^{-1})=\beta (1,\ldots ,1,g_{n+1}h_{n+1}g_{n+1}^{-1}h_{n+1}^{-1})$$
using $\beta (g)=1$. That is, $\beta$ restricted to the $(n+1)$-st factor has nontrivial kernel. Since that factor is simple, the restriction of $\beta$ to the $(n+1)$-st factor has to be trivial, so $\beta$ factors through the projection onto the first $n$ factors, hence (by induction hypothesis) onto one of them.
\end{proof}

\section{Semi-Arithmetic Groups and their Congruence Subgroups}\label{SectionSemiArithmetic}

\noindent Let $\Gamma\subset\PSL (2,\bbr )$ be a lattice and let $\tilde{\Gamma }$ be its preimage in $\SL (2,\bbr )$. By $\Gamma^{(2)}$ we denote the subgroup of $\Gamma$ generated by all $\gamma^2$ with $\gamma\in\Gamma$. Since $\Gamma$ is finitely generated, $\Gamma^{(2)}$ is then a normal subgroup of finite index in $\Gamma$.
\begin{Definition}\label{DefTFC}
The \emph{trace field} of $\Gamma$ is the field $\bbq (\tr\Gamma )\subset\bbr$ generated by all $\tr\gamma$ with $\gamma\in\tilde{\Gamma }$. The \emph{invariant trace field} of $\Gamma$ is the trace field of $\Gamma^{(2)}$.

A lattice $\Gamma$ satisfies the \emph{trace field condition} if its trace field and its invariant trace field agree.
\end{Definition}
Clearly the trace field contains the invariant trace field, but the two are not always equal. As the name suggests, the invariant trace field is the more useful invariant: commensurable lattices have the same invariant trace field, see \cite[Theorem 3.3.4]{MR1937957}, but not necessarily the same trace field.
Hence, if $\Gamma$ is any lattice then $\Gamma^{(2)}$ satisfies the trace field condition. Therefore any lattice has a finite index normal sublattice which satisfies the trace field condition.
\begin{Definition}
A lattice $\Gamma\subset\PSL (2,\bbr )$ is called \emph{semi-arithmetic} if its invariant trace field is a totally real number field and every trace $\tr\gamma$ for $\gamma\in\tilde{\Gamma }$ is an algebraic integer.\footnote{It follows from \cite[Lemma 3.5.6]{MR1937957} that this is equivalent to the definition given in the introduction.}
\end{Definition}
Being semi-arithmetic is stable under commensurability, therefore every semi-arithmetic lattice contains a semi-arithmetic lattice satisfying the trace field condition. 
For the following constructions let $\Gamma$ be a semi-arithmetic lattice satisfying the trace field condition, and let $k=\bbq (\tr\gamma )$. Then the sub-$k$-vector space $B=k[\Gamma ]$ of $\mathrm{M}(2,\bbr )$ generated by $\tilde{\Gamma}$ is in fact a sub-$k$-algebra, more precisely a quaternion algebra over $k$. The sub-$\mathfrak{o}_k$-algebra $\mathfrak{o}_k[\tilde{\Gamma }]$ of $B$ generated by $\tilde{\Gamma }$ is an order in $B$, though not necessarily a maximal one. We choose a maximal order $\mathrsfs{O}\supseteq\mathfrak{o}_k[\tilde{\Gamma }]$.

If $\mathrsfs{O}^1$ denotes the subgroup of $\mathrsfs{O}^{\times }$ consisting of elements with reduced norm one, $\tilde{\Gamma }$ becomes a subgroup of $\mathrsfs{O}^1$. Also write $\PO =\mathrsfs{O}^1/\{\pm\mathbf{1}\}$ so that $\Gamma$ is a subgroup of $\PO$.
\begin{Proposition}\label{EquivalentCharacterisationsArithmetic}
Let $\Gamma\subset\PSL (2,\bbr )$ be a semi-arithmetic lattice satisfying the trace field condition. Then the following are equivalent:
\begin{enumerate}
\item $\Gamma$ is arithmetic.
\item Let $k=\bbq (\tr\Gamma )\subset\bbr$. Then for every embedding $\sigma\colon k\to\bbr$ other than the identity inclusion and every $\gamma\in\tilde{\Gamma }$ one has $\lvert\sigma (\tr\gamma )\rvert\le 2$.
\item For every embedding $\sigma\colon k\to\bbr$ other than the identity inclusion, $B\otimes_{k,\sigma }\bbr$ is isomorphic to Hamilton's quaternions $\bbh$.
\item $\PO$ is a discrete subgroup of $\PSL (2,\bbr )$.
\item The index $(\PO :\Gamma )$ is finite.
\end{enumerate}
\end{Proposition}
\begin{proof}
The equivalence (i) $\Leftrightarrow$ (ii) is the main result in \cite{MR0398991}; the other equivalences follow from the explicit classification of arithmetic lattices in $\PSL (2,\bbr )$, see e.g.\ \cite[chapter 5]{MR1177168} or \cite[chapter 8]{MR1937957}.
\end{proof}
Now we discuss congruence subgroups. For an elementary definition, let $\Gamma\subset\PSL (2,\bbr )$ be a semi-arithmetic lattice satisfying the trace field condition, and let $k$ and $\mathrsfs{O}$ be as above. Then every nonzero ideal $\mathfrak{a}$ of $\mathfrak{o}_k$ defines a subgroup
$$\tilde{\Gamma }(\mathfrak{a})=\{ \gamma\in\tilde{\Gamma }\mid\gamma-\mathbf{1}\in\mathfrak{a}\cdot\mathrsfs{O}\}$$
and its image $\Gamma (\mathfrak{a})$ in $\Gamma$, called the \emph{principal congruence subgroup} of level $\mathfrak{a}$. A \emph{congruence subgroup} of $\Gamma$ is then a subgroup containing some principal congruence subgroup. Similarly we define principal congruence subgroups $\mathrsfs{O}^1(\mathfrak{a})$ and congruence subgroups of $\mathrsfs{O}^1$.

These groups can also defined more abstractly using algebraic groups: there is a canonical linear algebraic group $H$ over $k$ with $H(k)=B^1$; we may define it functorially by setting $H(A)=(B\otimes_kA)^1$ for every $k$-algebra $A$. Then $H$ is a twisted form of $\SL (2)_k$. By Weil restriction of scalars we obtain an algebraic group $G=\operatorname{Res}_{k|\bbq }H$ with a canonical identification $G(\bbq )=H(k)=B^1$. Then $G$ is a twisted form of $\SL (2)_{\bbq }^d$ where $d=[k:\bbq ]$; in particular $G(\bbc )$ is isomorphic to $\SL (2,\bbc )^d$.

Choosing a faithful representation $G\to\GL (n)$ we can define a congruence subgroup in $G(\bbq )$ to be one that contains the preimage of a congruence subgroup of $\GL (n,\bbz )$ as a finite index subgroup. This notion of congruence subgroup is independent of the representation $G\to\GL (n)$, see \cite[Proposition 4.1]{MR2192012}; that it is equivalent to the more elementary one given before follows by taking the representation of $G\to\GL (4d)$ by left multiplication on $B$, the latter considered as a $(4d)$-dimensional $\bbq$-vector space with the lattice $\mathrsfs{O}$.

Let $\bba^f$ be the ring of finite ad\`{e}les of $\bbq$ and endow $G(\bba^f)$ with the ad\`{e}lic topology. Similarly let $\bba_k^f$ be the ring of finite ad\`{e}les of $k$, then there is a canonical isomorphism $\bba^f\otimes_{\bbq }k=\bba_k^f$ inducing $G(\bba^f)=H(\bba_k^f)$. The closure of $\mathrsfs{O}^1$ in $G(\bba^f)$ can be identified with the completion of $\mathrsfs{O}^1$ with respect to all congruence subgroups; equivalently, with the group of elements of reduced norm one in the profinite completion of $\mathrsfs{O}$. Therefore we denote it by $\widehat{\mathrsfs{O}}^1$. It is a maximal compact open subgroup of $G(\bba^f)$.

There is a canonical bijection between open subgroups of $\widehat{\mathrsfs{O}}^1$ and congruence subgroups of $\mathrsfs{O}^1$: with a congruence subgroup of $\mathrsfs{O}^1$ we associate its closure in $G(\bba^f)$, and with an open subgroup of $\widehat{\mathrsfs{O}}^1$ we associate its intersection with $\mathrsfs{O}^1$. For the proof see again \cite[Proposition 4.1]{MR2192012}.

\begin{Proposition}[Strong Approximation for Semi-Arithmetic Groups]\label{StrongApproxSemiArithm}
The closure of $\tilde{\Gamma }$ in $G(\bba^f)=H(\bba_k^f)$ is open.
\end{Proposition}
\begin{proof}
First we claim that $\tilde{\Gamma }$ is Zariski-dense in $G$. It suffices to show that $\tilde{\Gamma }$ is Zariski-dense in $G(\bbc )\simeq\SL (2,\bbc )^d$, and the proof of an analogous but more complicated statement over the reals \cite[Proposition 2.1 and Corollary 2.2]{MR2968231} carries over mutatis mutandis.

Then we use a special case of a result of M.\ Nori \cite[Theorem 5.4]{MR880952}, see also \cite{MR735226}: if $G$ is an algebraic group over $\bbq$ such that $G(\bbc )$ is connected and simply connected (which is the case for our $G$ since $\pi_1(\SL (2,\bbc ))=\pi_1(\operatorname{SU}(2))=\pi_1(S^3)=1$) and $\Gamma$ is a finitely generated Zariski-dense subgroup of $G(\bbq )$ contained in some arithmetic subgroup of $G$, then the closure of $\Gamma$ in $G(\bba^f)$ is open.
\end{proof}
\begin{Proposition}\label{CorollaryGammaLocallyCongruence}
There exists a nonzero ideal $\mathfrak{m}$ of $\mathfrak{o}_k$, depending on $\Gamma$, such that for every ideal $\mathfrak{a}$ of $\mathfrak{o}_k$ prime to $\mathfrak{m}$ the homomorphism
$$\tilde{\Gamma }\hookrightarrow\mathrsfs{O}^1\twoheadrightarrow\mathrsfs{O}^1/\mathrsfs{O}^1(\mathfrak{a})$$
is surjective, i.e.\ the canonical homomorphism
$$\Gamma /\Gamma (\mathfrak{a})\to\PO /\PO (\mathfrak{a})$$
is an isomorphism of finite groups.
\end{Proposition}
The proof uses several results that will be used later on, so we mention them separately.
\begin{Theorem}[Strong Approximation for Quaternion Algebras]
$G(\bbq )=H(k)$ is dense in $G(\bba^f)=H(\bba_k^f)$.\footnote{Usually this result is phrased differently: if $\bba =\bba^f\times\bbr$ denotes the full ad\`{e}le ring, then $G(\bbq )\cdot G(\bbr )$ is dense in $G(\bba )$. But the latter is canonically isomorphic to $G(\bba^f)\times G(\bbr )$ which shows the equivalence to our formulation.}
\end{Theorem}
For the proof see e.g.\ \cite[Theorem 7.12]{MR1278263}.

We shall now investigate the quotient groups $\mathrsfs{O}^1/\mathrsfs{O}^1(\mathfrak{a})$. These are best understood locally: if $\mathfrak{p}$ is a finite prime of $k$, we set $\mathrsfs{O}_{\mathfrak{p}}=\mathrsfs{O}\otimes_{\mathfrak{o}_k}\mathfrak{o}_{\mathfrak{p}}$. We can then consider the group $\mathrsfs{O}_{\mathfrak{p}}^1$ of its elements of norm one, and its congruence subgroups $\mathrsfs{O}_{\mathfrak{p}}^1(\hat{\mathfrak{p}}^r)$. Recall that $\mathrsfs{O}_{\mathfrak{p}}$ is a maximal order in $B_{\mathfrak{p}}$.
\begin{Proposition}
Let $\mathfrak{a}$ be an ideal of $k$ with prime factorisation $\mathfrak{a}=\mathfrak{p}_1^{r_1}\mathfrak{p}_2^{r_2}\cdots\mathfrak{p}_n^{r_n}$. Then the canonical homomorphism
\begin{equation}\label{CongruenceQuotientPrimeDecomposition}
\mathrsfs{O}^1/\mathrsfs{O}^1(\mathfrak{a})\to\prod_{j=1}^n\mathrsfs{O}_{\mathfrak{p}_j}^1/\mathrsfs{O}_{\mathfrak{p}_j}^1(\hat{\mathfrak{p}}_j^{r_j})
\end{equation}
is an isomorphism of groups.
\end{Proposition}
\begin{proof}
Injectivity is easy, so we only show surjectivity.

We use the description of $H(\bba_k^f)$ as the restricted direct product of the completions $B_{\mathfrak{l}}^1=(B\otimes_kk_{\mathfrak{l}})^1$, restricted with respect to the compact subgroups $\mathrsfs{O}_{\mathfrak{l}}^1$. For $j=1,\ldots ,n$ take an element $x_j\in\mathrsfs{O}_{\mathfrak{p}_j}^1$. The Strong Approximation Theorem furnishes us with an element $\beta\in H(k)= B^1$ with the following properties:
\begin{itemize}
\item For $j=1,\ldots ,n$, $\beta$ considered as an element of $B_{\mathfrak{p}_j}^1$ is congruent to $x_j$ modulo $\mathrsfs{O}_{\mathfrak{p}_j}^1(\hat{\mathfrak{p}_j}^{r_j})$ (note that the latter is an open subgroup of $B_{\mathfrak{p}_j}^1$).
\item For each finite prime $\mathfrak{l}$ different from all $\mathfrak{p}_j$'s, $\beta$ is in $\mathrsfs{O}_{\mathfrak{l}}^1$.
\end{itemize} 
Then $\beta\in\mathrsfs{O}^1$, and its class in the left hand side of (\ref{CongruenceQuotientPrimeDecomposition}) maps to $(x_1,\ldots ,x_n)$.
\end{proof}
Note that our proof also shows that the map
\begin{equation*}
\mathrsfs{O}^1/\mathrsfs{O}^1(\mathfrak{a})\to\prod_{j=1}^n\mathrsfs{O}^1/\mathrsfs{O}^1(\mathfrak{p}_j^{r_j})
\end{equation*}
is an isomorphism.
\begin{Corollary}\label{ChineseRemainderForPSL}
The canonical homomorphism
\begin{equation*}
\PO /\PO (\mathfrak{a})\to\prod_{j=1}^n\PO /\PO (\mathfrak{p}_j^{r_j})
\end{equation*}
is an epimorphism whose kernel is isomorphic to $(\bbz /2\bbz )^{d}$ for some $d<n$.
\end{Corollary}
\begin{proof}
The homomorphism $\mathrsfs{O}^1/\mathrsfs{O}^1(\mathfrak{p}_j^{r_j})\to\PO /\PO (\mathfrak{p}_j^{r_j})$ is always surjective, and it is injective precisely when $\mathfrak{p}_j^{r_j}$ divides $(2)$, otherwise it has kernel isomorphic to $\bbz /2\bbz$. Similarly the kernel of $\mathrsfs{O}^1/\mathrsfs{O}^1(\mathfrak{a})\to\PO /\PO (\mathfrak{a})$ is either trivial or $\bbz /2\bbz$. So the corollary follows from the remark preceding it.
\end{proof}
\begin{Corollary}\label{SurjectiveOnCoprime}
Let $\Delta\subseteq\mathrsfs{O}^1$ be a congruence subgroup, containing $\mathrsfs{O}^1(\mathfrak{m})$ for some ideal $\mathfrak{m}$ of $k$. Let $\mathfrak{a}$ be an ideal of $k$ which is coprime to $\mathfrak{m}$. Then the composition
$$\Delta\hookrightarrow\mathrsfs{O}^1\twoheadrightarrow\mathrsfs{O}^1/\mathrsfs{O}^1(\mathfrak{a})$$
is surjective.
\end{Corollary}
\begin{proof}
This is equivalent to the statement $\mathrsfs{O}^1(\mathfrak{m})\cdot\mathrsfs{O}^1(\mathfrak{a})=\mathrsfs{O}^1$, and this in turn follows from the isomorphism of finite groups
\begin{equation*}\mathrsfs{O}^1/(\mathrsfs{O}^1(\mathfrak{m})\cap\mathrsfs{O}^1(\mathfrak{a}))\to\mathrsfs{O}^1/\mathrsfs{O}^1(\mathfrak{m})\times\mathrsfs{O}^1/\mathrsfs{O}^1(\mathfrak{a}).\qedhere\end{equation*}
\end{proof}
\begin{proof}[Proof of Proposition \ref{CorollaryGammaLocallyCongruence}]
By Proposition \ref{StrongApproxSemiArithm} there exists some ideal $\mathfrak{m}$ of $k$ with $\mathrsfs{O}^1(\mathfrak{m})\subseteq\overline{\tilde{\Gamma}}$, where the latter denotes the closure of $\tilde{\Gamma }$ in $\widehat{\mathrsfs{O}}^1\subset G(\bba^f)$. This does the job by Corollary \ref{SurjectiveOnCoprime}.
\end{proof}
\begin{Corollary}\label{CorChangingTheLevel}
Let $\mathfrak{a}$ and $\mathfrak{b}$ be two coprime ideals of $k$ which are both prime to $2$. Then the canonical homomorphism
\begin{equation*}
\PO (\mathfrak{a})/\PO (\mathfrak{ab})\to \PO / \PO (\mathfrak{b})
\end{equation*}
is an isomorphism.\hfill $\square$
\end{Corollary}

\section{Congruence Quotients of Semi-Arithmetic Groups}

\noindent Our next step is to determine the quotients on the right hand side of (\ref{CongruenceQuotientPrimeDecomposition}). This is done by distinguishing between the ramified and the unramified case. To simplify notation, let $K$ be a $p$-adic field with ring of integers $\mathfrak{o}_K$ and prime ideal $\mathfrak{p}=(\pi )$. Let $q=p^f$ be the cardinality of the residue class field $\kappa =\mathfrak{o}_K/\mathfrak{p}$. Let $B$ be an unramified quaternion algebra over $K$, and let $\mathrsfs{O}\subset B$ be a maximal order. We may assume that $B=\mathrm{M}(2,K)$ and $\mathrsfs{O}=\mathrm{M}(2,\mathfrak{o}_K)$; then $\mathrsfs{O}^1=\SL (2,\mathfrak{o}_K)$ and $\mathrsfs{O}^1(\mathfrak{p})$ is the kernel of the reduction map $\SL (2,\mathfrak{o}_K)\to\SL (2,\kappa )$.
\begin{Lemma}\label{OoneModOonepIsSLTwo}
Let $r\ge 1$. The reduction map $\SL (2,\mathfrak{o}_K)\to\SL (2,\mathfrak{o}_K/\mathfrak{p}^r)$ is surjective and thus induces an isomorphism $\mathrsfs{O}^1/\mathrsfs{O}^1(\mathfrak{p}^r)\to\SL (2,\mathfrak{o}_K/\mathfrak{p}^r)$. In particular $\mathrsfs{O}^1/\mathrsfs{O}^1(\mathfrak{p})$ is isomorphic to $\SL (2,q)$.
\end{Lemma}
\begin{proof}
Let
$$\overline{\gamma }=\begin{pmatrix}\overline{a} &\overline{b} \\ \overline{c} &\overline{d}\end{pmatrix}\in\SL (2,\mathfrak{o}_K/\mathfrak{p}^r)$$
and lift $\overline{\gamma }$ arbitrarily to a matrix
$$\gamma =\begin{pmatrix} a&b\\ c&d\end{pmatrix} \in\GL (2,\mathfrak{o}_K).$$
The determinant $\delta =\det\gamma$ is an element of $1+\mathfrak{p}^r$, hence so is its inverse $\frac{1}{\delta }$. Therefore
$$\gamma ' =\begin{pmatrix}\frac{a}{\delta }&\frac{b}{\delta }\\ c&d\end{pmatrix}\in\SL (2,\mathfrak{o}_K)$$
still reduces to $\overline{\gamma }$.
\end{proof}
\begin{Lemma}\label{HigherSubquotientsSolvable}
Let $r\ge 1$. Assumptions as before, the quotient $\mathrsfs{O}^1(\mathfrak{p}^r)/\mathrsfs{O}^1(\mathfrak{p}^{r+1})$ is isomorphic to $(\bbz /p\bbz )^{3f}$.
\end{Lemma}
\begin{proof}
We construct a map
\begin{equation*}
(\mathrsfs{O}/\mathfrak{p}\mathrsfs{O})_0\to \SL (2,\mathfrak{o}_K/\mathfrak{p}^{r+1}),\quad [A]\mapsto [1+\pi^rA].
\end{equation*}
Here the left hand side denotes the subgroup of those elements of $\mathrsfs{O}/\mathfrak{p}\mathrsfs{O}=M(2,\kappa )$ that have trace $\equiv 0\bmod\mathfrak{p}$. Note that $\det (1+\pi^rA)\equiv1+\pi^r\tr A\bmod\mathfrak{p}^{r+1}$, so the map is indeed well-defined. It is an injective group homomorphism, and its image is precisely the image of $\mathrsfs{O}^1(\mathfrak{p}^r)$ in $\SL (2,\mathfrak{o}_K/\mathfrak{p}^{r+1})$, which is isomorphic to $\mathrsfs{O}^1(\mathfrak{p}^r)/\mathrsfs{O}^1(\mathfrak{p}^{r+1})$.
\end{proof}
Now we turn to the ramified case. We use the explicit descritption of $B$ and $\mathrsfs{O}$ given in \cite[section 6.4]{MR1937957}. Let $L|K$ be the unique unramified quadratic extension, then $B$ is up to isomorphism given by
\begin{equation*}
B=\left\{\begin{pmatrix}a& b\\ \pi b' &a'\end{pmatrix}\middle\lvert a,b\in L\right\} ,
\end{equation*}
where $a\mapsto a'$ is the nontrivial element of $\operatorname{Gal}(L|K)$. This contains a unique maximal order,
\begin{equation*}
\mathrsfs{O}=\left\{\begin{pmatrix}a& b\\ \pi b' &a'\end{pmatrix}\middle\lvert a,b\in \mathfrak{o}_L\right\} ,
\end{equation*}
and $\mathrsfs{O}$ has a unique maximal two-sided ideal,
\begin{equation*}
\mathrsfs{M}=\begin{pmatrix}0&1\\ \pi & 0\end{pmatrix}\mathrsfs{O}=\left\{\begin{pmatrix}\pi a& b\\ \pi b' &\pi a'\end{pmatrix}\middle\lvert a,b\in\mathfrak{o}_L\right\} .
\end{equation*}
It satisfies $\mathrsfs{M}^2=\mathfrak{p}\mathrsfs{O}$. We define congruence subgroups $\mathrsfs{O}^1(\mathrsfs{M}^r)=\mathrsfs{O}^1\cap (1+\mathrsfs{M}^r)$, so that $\mathrsfs{O}^1(\mathfrak{p}^r)=\mathrsfs{O}^1(\mathrsfs{M}^{2r})$.
\begin{Lemma}
The quotient $\mathrsfs{O}^1/\mathrsfs{O}^1(\mathrsfs{M})$ is a cyclic group of order $q+1$.
\end{Lemma}
\begin{proof}
Since $L|K$ is unramified, the quotient $\lambda =\mathfrak{o}_L/\pi\mathfrak{o}_L$ is a finite field of order $q^2$. We construct a map
\begin{equation*}
\mathrsfs{O}^1/\mathrsfs{O}^1(\mathrsfs{M})\to\lambda^{\times },\quad \left[\begin{pmatrix} a&b\\ \pi b'& a'\end{pmatrix}\right] \mapsto a\bmod\pi .
\end{equation*}
This is easily seen to be an injective group homomorphism whose image is the kernel of the norm map $N_{\lambda\mid\kappa }$. That norm map is surjective to $\kappa^{\times }$, so its kernel has order $(q^2-1)/(q-1)=q+1$.
\end{proof}
\begin{Lemma}
Let $r\ge 1$. Then $\mathrsfs{O}^1(\mathrsfs{M}^r)/\mathrsfs{O}^1(\mathrsfs{M}^{r+1})$ is isomorphic to the additive group of $\kappa$.
\end{Lemma}
\begin{proof}
We construct injective group homomorphisms
\begin{equation*}
\mathrsfs{O}^1(\mathrsfs{M}^{2r})/\mathrsfs{O}^1(\mathrsfs{M}^{2r+1})\to\lambda ,\quad \left[\begin{pmatrix} a&b\\ \pi b'& a'\end{pmatrix}\right] \mapsto \frac{a-1}{\pi^r }\bmod\pi
\end{equation*}
and
\begin{equation*}
\mathrsfs{O}^1(\mathrsfs{M}^{2r-1})/\mathrsfs{O}^1(\mathrsfs{M}^{2r})\to\lambda ,\quad \left[\begin{pmatrix} a&b\\ \pi b'& a'\end{pmatrix}\right] \mapsto \frac{b}{\pi^{r-1}}\bmod\pi .
\end{equation*}
The image is in both cases the kernel of the trace map $\tr_{\lambda\mid\kappa }$.
\end{proof}
We summarise the results, reformulated for number fields:
\begin{Corollary}
Let $k$ be a number field and $B$ a quaternion algebra over $k$, unramified over at least one infinite place of $k$. Let $\mathrsfs{O}\subset B$ be a maximal order and let $\mathfrak{p}$ be a prime of $k$ of norm $q=p^f$. Let $r\ge 1$ and $H=\mathrsfs{O}^1/\mathrsfs{O}^1(\mathfrak{p}^r)$.
\begin{enumerate}
\item If $B$ is ramified at $\mathfrak{p}$, then $H$ is solvable; the prime numbers appearing as orders in its composition series are $p$ and the prime divisors of $q+1$.
\item If $B$ is unramified at $\mathfrak{p}$ and $\mathfrak{p}\nmid 6$, then $H$ is not solvable. Its composition factors are: once $\bbz /2\bbz$, once $\PSL (2,q)$ and $3f(r-1)$ times $\bbz /p\bbz$.
\end{enumerate}
\end{Corollary}
In case (ii) for $\mathfrak{p}\mid 6$ we have to replace $\PSL (2,q)$, which is not necessarily simple then, by its composition factors.

\section{Characters for Fuchsian Groups}

\noindent In this section we prove a criterion for two isomorphic lattices in $\PSL (2,\bbr )$ being conjugate:

\begin{Theorem}\label{CullerShalenForPSLTwoR}
Let $\Gamma$ be a group, and for $j=1,2$ let $\varrho_j\colon\Gamma\to\PSL (2,\bbr )$ be an injective group homomorphism such that $\varrho_j(\Gamma )$ is a lattice. Let $\Delta\subseteq\Gamma$ be a finite index subgroup, and assume that
\begin{equation}
\tr^2\varrho_1(\gamma )=\tr^2\varrho_2(\gamma )\text{ for all }\gamma\in\Delta .
\end{equation}
Then there exists a unique $a\in\PGL (2,\bbr )$ such that $\varrho _2(\gamma )=a\varrho_1(\gamma )a^{-1}$ for all $\gamma\in\Gamma $.
\end{Theorem}
The proof of Theorem \ref{CullerShalenForPSLTwoR} rests on the following result, see \cite[Proposition 1.5.2]{MR683804}, as well as on subsequent elementary lemmas.
\begin{Theorem}[Culler--Shalen]\label{CullerShalenOriginal}
Let $\varrho_1 ,\varrho_2\colon\Gamma\to\SL (2,\bbc )$ be two representations such that
\begin{equation}
\tr\varrho_1(\gamma )=\tr\varrho_2(\gamma )\text{ for every }\gamma\in\Gamma ,
\end{equation}
and assume that $\varrho_1$ is irreducible. Then there exists $a\in\SL (2,\bbc )$, unique up to sign, such that $\varrho_2(\gamma )=a\varrho_1(\gamma )a^{-1}$ for every $\gamma\in\Gamma$.
\end{Theorem}
\begin{Lemma}
Let $g\in\PSL (2,\bbr )$ and let $\Sigma\subset\PSL (2,\bbr )$ be a group generated by two hyperbolic elements without common fixed points. Then there exists $s\in \Sigma$ with $sg$ hyperbolic.
\end{Lemma}
\begin{proof}
Lift $g$ to an element $G\in\SL (2,\bbr )$. First we will show that there exists some $S\in\tilde{\Sigma }$ with $\tr (SG)\neq 0$.

Assume, on the contrary, that $\tr (SG)=0$ for all $S\in\tilde{\Sigma }$. Choose two hyperbolic elements $S_1,S_2\in\tilde{\Sigma}$ without common fixed points; without loss of generality we may assume that
$$S_1=\begin{pmatrix}\lambda & 0\\ 0&\lambda^{-1}\end{pmatrix},\quad S_2=\begin{pmatrix}w&x\\ y&z\end{pmatrix},\quad G=\begin{pmatrix}a&b\\ c&d\end{pmatrix}$$
for some $\lambda >1$ and $xy\neq 0$. Then
$$\lambda a+\lambda^{-1}d=\tr (S_1G)=0=\tr (G)=a+d,$$
hence $a=d=0$ and
$$G=\begin{pmatrix}0 & b\\ c & 0\end{pmatrix},\quad bc=-\det (G)=-1,\text{ so }b,c\neq 0.$$
But then
$$cx+by=\tr (S_2G)=0=\tr (S_1S_2G)=\lambda cx+\lambda^{-1}by,$$
hence $cx=by=0$; but we know that $b,c,x,y\neq 0$, contradiction.

So there exists some $S\in\tilde{\Sigma }$ with $\tr (SG)\neq 0$; without loss of generality we assume that already $\tr G\neq 0$. Take some arbitrary hyperbolic $T\in\tilde{\Sigma}$; by the elementary equation
\begin{equation}\label{TraceRelationAB}
\tr (AB)+\tr (AB^{-1})=\tr (A)\cdot \tr (B)\text{ for all }A,B\in\SL (2,\bbc )
\end{equation}
then
$$\lvert\tr (T^NG)\rvert +\lvert\tr (T^{-N}G)\rvert\ge\lvert\tr (T^NG)+\tr (T^{-N}G)\rvert =\lvert\tr (T^N)\tr (G)\rvert .$$
But the right hand side goes to $\infty$ as $N\to\infty$, so for sufficiently large $N$, at least one of $\lvert\tr (T^NG)\rvert$ and $\lvert\tr (T^{-N}G)\rvert$ must be larger than $2$.
\end{proof}
\begin{Lemma}\label{FuchsiansAreGeneratedByHyberbolics}
Let $\Gamma\subset\PSL (2,\bbr )$ be a lattice. Then there exists a finite generating system of $\Gamma$ only consisting of hyperbolic elements.
\end{Lemma}
\begin{proof}
Assume that $\Gamma$ is generated by $g_1,\ldots ,g_n$. By \cite[Exercise 2.13]{MR1177168}, $\Gamma$ contains two hyperbolic elements $h_1,h_2$ without common fixed points; let them generate the group $S$. For each $1\le j\le n$ choose some $s_j\in S$ with $s_jg_j$ hyperbolic. Then $\Gamma$ is generated by the hyperbolic elements $h_1,h_2,s_1g_1,\ldots ,s_ng_n$.
\end{proof}
\begin{Lemma}\label{ReductionToRealCaseConjugation}
Let $a\in\SL (2,\bbc )$ and let $\Gamma\subset\SL (2,\bbr )$ be a lattice with $a\Gamma a^{-1}\subset\SL (2,\bbr )$. Then $a\in\bbc^{\times }\cdot\GL (2,\bbr )$.
\end{Lemma}
\begin{proof}
Since $\Gamma$ is Zariski-dense in $\SL (2,\bbr )$ we may deduce that $a\SL (2,\bbr )a^{-1}\subseteq\SL (2,\bbr )$. The sub-$\bbr$-vector space of $\mathrm{M}(2,\bbc )$ generated by $\SL (2,\bbr )$ is $\mathrm{M}(2,\bbr )$, so $a\mathrm{M}(2,\bbr )a^{-1}=\mathrm{M}(2,\bbr )$. By the Skolem--Noether Theorem, the automorphism $g\mapsto aga^{-1}$ of $\mathrm{M}(2,\bbr )$ has to be an inner automorphism, i.e.\ there exists $b\in\GL (2,\bbr )$ with $aga^{-1}=bgb^{-1}$ for all $g\in\mathrm{M}(2,\bbr )$ and hence, by linear extension, also for all $g\in\mathrm{M} (2,\bbc )$. But this means that $ba^{-1}$ is in the centre of $\mathrm{M}(2,\bbc )$ which is $\bbc^{\times }$.
\end{proof}
\begin{proof}[Proof of Theorem \ref{CullerShalenForPSLTwoR}]
Without loss of generality we may assume that $\Delta$ is torsion-free by Selberg's Lemma \cite[Lemma 8]{MR0130324}, hence it has a presentation
$$\Delta =\langle g_1,\ldots ,g_m\mid [g_1,g_{n+1}][g_2,g_{n+2}]\cdots [g_n,g_{2n}]=1\rangle \text{ with }m=2n$$
(in the cocompact case), or is free on some generators $g_1,\ldots ,g_m$ (otherwise).
 By \cite[Theorem 4.1]{MR1124819} each $\varrho_j|_{\Delta}$ can be lifted to representations $\tilde{\varrho }_j\colon\Delta\to\SL (2,\bbr )$; furthermore again by that theorem we can arbitrarily prescribe the sign of each lift of $\varrho_j(g_i)$, so we may assume that
\begin{equation}\label{SameTracesForGenerators}
\tr \tilde{\varrho }_1(g_i)=\tr\tilde{\varrho }_2(g_i)\text{ for all }1\le i\le m.
\end{equation}
More generally,
$$\tr \tilde{\varrho }_1(\gamma )=\varepsilon (\gamma )\cdot\tr \tilde{\varrho }_2(\gamma )\text{ for all }\gamma\in\Delta ,$$
where $\varepsilon$ is some function $\Delta\to\{\pm 1\}$. Note that $\varepsilon$ is uniquely determined by this equation because the traces cannot be zero since elements of $\varrho_j(\Delta )$ are not elliptic. Furthermore $\varepsilon (g_i)=1$ for every generator $g_i$ by (\ref{SameTracesForGenerators}).

We now show that $\varepsilon$ is identically $1$. The crucial step is the following implication:
\begin{equation}\label{InductionStepForEpsilon}
\text{If }\varepsilon (\gamma )=\varepsilon (\delta )=1,\text{ then }\varepsilon (\gamma\delta )=\varepsilon (\gamma\delta^{-1} )=1.
\end{equation}
So assume that $\varepsilon (\gamma )=\varepsilon (\delta )=1$. We deduce from (\ref{TraceRelationAB}):
\begin{equation}\label{TraceEquationForXandY}
\begin{gathered}
\varepsilon (\gamma\delta )\tr \tilde{\varrho }_1(\gamma\delta ) +\varepsilon (\gamma\delta^{-1})\tr \tilde{\varrho }_1(\gamma\delta^{-1})=\tr\tilde{\varrho }_2(\gamma\delta )+\tr\tilde{\varrho }_2(\gamma\delta^{-1})\\
=(\tr\tilde{\varrho }_2(\gamma))\cdot (\tr\tilde{\varrho }_2(\delta ))=(\tr\tilde{\varrho }_1(\gamma))\cdot (\tr\tilde{\varrho }_1(\delta ))=\tr\tilde{\varrho }_1(\gamma\delta )+\tr\tilde{\varrho }_1(\gamma\delta^{-1}).
\end{gathered}
\end{equation}
If $\varepsilon (\gamma\delta )$ and $\varepsilon (\gamma\delta^{-1})$ were both negative, (\ref{TraceEquationForXandY}) would entail that $(\tr\tilde{\varrho }_2(\gamma ))\cdot (\tr\tilde{\varrho }_2(\delta ))=0$ which is absurd because $\Delta$ does not contain elliptic elements. If $\varepsilon (\gamma\delta )=1$ and $\varepsilon (\gamma\delta^{-1})=-1$, then $\tr\tilde{\varrho }_2(\gamma\delta^{-1})=0$ which is again absurd; the other mixed case is ruled out in an analogous way. This proves (\ref{InductionStepForEpsilon}).

Now we can prove that $\varepsilon (\gamma )=1$ for every $\gamma\in\Delta$ using induction on the word length $\ell (\gamma )$: this is the number of factors $g_j^{\pm 1}$ needed to obtain $\gamma$ as a product. If $\ell (\gamma )=1$ then $\gamma =g_j^{\pm 1}$; since $\varepsilon (\gamma )=\varepsilon (\gamma^{-1})$, this must be equal to $\varepsilon (g_j)=1$. If $\varepsilon (\gamma )=1$ for all $\gamma$ with $\ell (\gamma )\le n$ we may use (\ref{InductionStepForEpsilon}) and the trivial identity $\varepsilon (\gamma^{-1})=\varepsilon (\gamma )$ to show the statement for all $\gamma$ with $\ell (\gamma )\le n+1$. Therefore by induction, $\varepsilon$ is identically $1$, hence
\begin{equation*}
\tr\tilde{\varrho }_1(\gamma )=\tr\tilde{\varrho }_2(\gamma )\text{ for all }\gamma\in\Delta .
\end{equation*}
By Theorem \ref{CullerShalenOriginal} this means that $\tilde{\varrho }_1$ is conjugate to $\tilde{\varrho }_2$ within $\SL (2,\bbc )$, but since all images are real, the conjugation must be possible within $\GL (2,\bbr )$ by Lemma \ref{ReductionToRealCaseConjugation}. This in turn means that $\varrho_1|_{\Delta }$ and $\varrho_2|_{\Delta }$ are conjugate in $\PGL (2,\bbr )$.

We need to extend this to the entire group $\Gamma$. Without loss of generality we may assume that $\varrho_1|_{\Delta }=\varrho_2|_{\Delta }$. By Lemma \ref{FuchsiansAreGeneratedByHyberbolics} there exists a generating system $\gamma_1,\ldots ,\gamma_m$ of $\Gamma$, not necessarily related in any way to that of $\Delta$, such that all $\varrho_1(\gamma_j)$ are hyperbolic. But some power of each $\gamma_j$ is contained in $\Delta$, and hence $\varrho_1(\gamma_j)^N=\varrho_2(\gamma_j)^N$. Under the assumptions on $\gamma_j$ this entails $\varrho_1(\gamma_j)=\varrho_2(\gamma_j)$, i.e.\ $\varrho_1=\varrho_2$.
\end{proof}

\section{Modular Embeddings}\label{SectionModularEmbedding}

\noindent Let once again $\Gamma\subset\PSL (2,\bbr )$ be a semi-arithmetic lattice satisfying the trace field property, with trace field $k$, quaternion algebra $B$, maximal order $\mathrsfs{O}$ and algebraic group $G=\operatorname{Res}_{k|\bbq }H$. As explained above, $\Gamma$ is a subgroup of the arithmetic group $\PO$. Now that latter group naturally lives on the symmetric space of $G$, i.e.\ on $G(\bbr )/K$ for a maximal compact subgroup $K$. This space can be described explicitly as $\mathfrak{H}^r$ where $\mathfrak{H}$ is the upper half-plane and $r\le d=[k:\bbq ]$. Let $\sigma_1,\ldots ,\sigma_d\colon k\to\bbr$ be the field embeddings, where $\sigma_1$ is the identity embedding. We also may assume that the quaternion algebra $B\otimes_{k,\sigma_i}\bbr$ is isomorphic to $\mathrm{M}(2,\bbr )$ for each $1\le i\le r$ and isomorphic to $\bbh$ for $r<i\le d$.

For each $1\le i\le r$ we choose an isomorphism $\alpha_i\colon B\otimes_{k,\sigma_i}\bbr\to\mathrm{M}(2,\bbr )$. We obtain an embedding
$$\alpha\colon\mathrsfs{O}^1\hookrightarrow\SL (2,\bbr )^r,\quad x \mapsto (\alpha_1(x),\ldots ,\alpha_r(x))$$
descending to an embedding $\alpha\colon\PO\hookrightarrow\PSL (2,\bbr )^r$. We denote the image by $\Lambda =\alpha (\PO )$.
\begin{Theorem}
$\Lambda$ is an irreducible arithmetic lattice in $\PSL (2,\bbr )^r$.
\end{Theorem}
For the proof see e.g.\ \cite{MR0204426}.

Note that $\alpha (\Gamma )$ becomes a subgroup of $\Lambda$. It has finite index precisely if $\Gamma$ is already arithmetic; in every case $\alpha (\Gamma )$ is a Zariski-dense subgroup of $\Lambda$ by the proof of Proposition \ref{StrongApproxSemiArithm}. Zariski-dense subgroups of infinite index in arithmetic groups are called \emph{thin}, and so we have shown:
\begin{Corollary}\label{SemiArithmeticAreThin}
If $\Gamma$ is not arithmetic itself, the embedding $\alpha\colon\Gamma\to\Lambda$ realises $\Gamma$ as a thin group.
\end{Corollary}
Let $\PSL (2,\bbr )^r$ operate by component-wise M\"{o}bius transformations on $\mathfrak{H}^r$; the induced action of $\Lambda$ on $\mathfrak{H}^r$ is properly discontinuous and has a quotient of finite volume. This motivates the following definition:
\begin{Definition}
A \emph{modular embedding} of $\Gamma$ is a holomorphic embedding $F\colon\mathfrak{H}\to\mathfrak{H}^r$ such that
$$F(\gamma\tau )=\alpha (\gamma )F(\tau )$$
for every $\gamma\in\Gamma$ and every $\tau\in\mathfrak{H}$.
\end{Definition}
The following result which will be used later on is \cite[Corollary 5]{MR1745404}:
\begin{Proposition}\label{ModularEmbeddingEnablesToPinDownTrace}
Let $\Gamma\subset\PSL (2,\bbr )$ be a semi-arithmetic group which satisfies the trace field property and admits a modular embedding, and let $k=\bbq (\tr\Gamma)$. Let $\gamma\in\tilde{\Gamma }$ be hyperbolic and let $\sigma\colon k\to\bbr$ be an embedding which is not the identity inclusion. Then $\lvert\sigma (\tr\gamma )\rvert <\lvert\tr\gamma\rvert$.
\end{Proposition}
Note that if $\Gamma$ is an arithmetic group then even $\lvert\sigma (\tr\gamma )\rvert <2$ by Proposition \ref{EquivalentCharacterisationsArithmetic}.

\section{Congruence Rigidity}\label{SectionCongruenceRigidity}

\noindent Let $\Gamma\subset\PSL (2,\bbr )$ be a semi-arithmetic lattice satisfying the trace field condition, with trace field $k=\bbq (\tr\Gamma )$. Let $B=k[\tilde{\Gamma }]$ be the associated quaternion algebra and $G$ the algebraic group over $\bbq$ with $G(\bbq )=B^1$. Let $\mathrsfs{O}\subset B$ be a maximal order containing $\tilde{\Gamma }$, and let $\mathfrak{m}\subset\mathfrak{o}_k$ be such that a finite index subgroup of $\Gamma$ is ad\`{e}lically dense in $\PO (\mathfrak{m})$; in particular, $\mathfrak{m}$ satisfies the conclusion of Proposition \ref{CorollaryGammaLocallyCongruence}.

For the statement of the next proposition, let $\mathfrak{m}=\mathfrak{l}_1^{r_1}\cdots\mathfrak{l}_n^{r_n}$ be the prime factorisation of $\mathfrak{m}$. Let $\ell_j$ be the norm of the prime ideal $\mathfrak{l}_j$. Then $S(\mathfrak{m})$ is the finite set of all rational primes diving some $\lvert\PSL (2,\ell_j)\rvert$ (this includes the primes dividing $\ell_j$ or $\ell_j+1$). Note that if $\mathfrak{m}'$ is an ideal which has the same prime divisors as $\mathfrak{m}$ and if $\ell$ is a rational prime dividing the order of $\PO /\PO (\mathfrak{m}')$, then $\ell\in S(\mathfrak{m})$. Also $S(6)$ is the set consisting of $2,3$ and all prime divisors of orders of $\PSL (2,q)$ where $q$ is the norm of a prime ideal $\mathfrak{p}$ in $k$ with $\mathfrak{p}\mid 6$. Finally $S(\Gamma )$ is the union of $S(\mathfrak{m})\cup S(6)$, the primes lying over the ramification divisor of $B$ and the primes that ramify in $k$. Still, $S(\Gamma )$ is a finite set of rational primes.
\begin{Proposition}\label{CongruenceQuotientsGivePrimes}
Let $\Gamma$ as above, and let $q=p^f$ be an odd prime power which is prime to all primes in $S(\Gamma )$. Let $\Delta\subset\Gamma$ be a normal congruence subgroup such that $\Gamma /\Delta\simeq\PSL (2,q)$. Then there exists a unique prime $\mathfrak{p}$ of norm $q$ in $k$ such that $\Delta =\Gamma (\mathfrak{p})$.
\end{Proposition}
\begin{proof}
There exists an ideal $\mathfrak{n}$ such that $\Delta\supseteq\Gamma (\mathfrak{n})$ and a finite index subgroup of $\Delta$ is ad\`{e}lically dense in $\PO (\mathfrak{n})$. We may assume that $\mathfrak{m}$ divides $\mathfrak{n}$. Write $\mathfrak{n}=\mathfrak{n}'\cdot\mathfrak{n}_{\mathfrak{m}}$ with $\mathfrak{n}'$ coprime to $\mathfrak{m}$ and $\mathfrak{n}_{\mathfrak{m}}$ having the same prime divisors as $\mathfrak{m}$; then $\Gamma$ also contains a subgroup which is ad\`{e}lically dense in  $\PO (\mathfrak{n}_{\mathfrak{m}})$. By Proposition \ref{CorollaryGammaLocallyCongruence} this entails that $\Gamma$ surjects onto $\PO /\PO (\mathfrak{n}')$.

Denote the quotient map modulo $\Delta$ by
\begin{equation*}
\pi\colon\Gamma\to\PSL (2,q).
\end{equation*}
Note that $\pi$ is continuous in the ad\`{e}lic topology on $\Gamma$ since it vanishes on $\Gamma (\mathfrak{n})$.

Now $\Gamma (\mathfrak{n}')=\Gamma\cap\PO (\mathfrak{n}')$ is a normal subgroup of $\Gamma$, hence its image under $\pi$ is a normal subgroup of $\PSL (2,q)$. Since that group is simple, the image can only be $\PSL (2,q)$ or the trivial group. Assume it were the entire group, then in the sequence
\begin{equation*}
\PSL (2,q)\twoheadleftarrow \Gamma (\mathfrak{n}')/\Gamma (\mathfrak{n})\hookrightarrow\PO (\mathfrak{n}')/\PO (\mathfrak{n})\simeq\PO /\PO (\mathfrak{n}_{\mathfrak{m}})
\end{equation*}
(where the isomorphism is by Corollary \ref{CorChangingTheLevel}) the order of the left hand side would divide the order of the right hand side. But the former is divisible by $p$, the latter only by primes in $S(\Gamma )$. A contradiction, hence the image of $\Gamma (\mathfrak{n}')$ under $\pi$ is the trivial group. In other words,
\begin{equation*}
\Delta\supseteq\Gamma (\mathfrak{n}').
\end{equation*}
This implies that $\pi$ descends to an epimorphism
\begin{equation*}
\pi\colon\Gamma /\Gamma (\mathfrak{n}')\twoheadrightarrow\PSL (2,q).
\end{equation*}
By Proposition \ref{CorollaryGammaLocallyCongruence} the inclusion $\Gamma\subseteq\PO$ induces an isomorphism
\begin{equation*}
\alpha\colon\Gamma /\Gamma (\mathfrak{n}')\overset{\simeq}{\longrightarrow}\PO /\PO (\mathfrak{n}').
\end{equation*}
So by composition we obtain an epimorphism $\pi\circ\alpha^{-1}\colon\PO /\PO (\mathfrak{n}')\twoheadrightarrow\PSL (2,q)$. Let $\mathfrak{n}'=\mathfrak{p}_1^{r_1}\cdots\mathfrak{p}_n^{r_n}$ with distinct prime ideals $\mathfrak{p}_j$, and let $\operatorname{rad}(\mathfrak{n}')=\mathfrak{p}_1\cdots\mathfrak{p}_n$. Then $\PO (\operatorname{rad}(\mathfrak{n}'))/\PO (\mathfrak{n}')$ is a solvable normal subgroup of $\PO /\PO (\mathfrak{n}')$ by Lemma \ref{HigherSubquotientsSolvable}, so its image by $\pi\circ\alpha^{-1}$ has to be a solvable normal subgroup of $\PSL (2,q)$, i.e.\ trivial. Therefore $\pi\circ\alpha^{-1}$ factors through $\PO /\PO (\operatorname{rad}(\mathfrak{n}'))$; we summarise this in a diagram:
\begin{equation}\label{DiagramWithDashedEpi}
\xymatrix{
\Gamma /\Gamma (\mathfrak{n}') \ar[r]^{\simeq\hspace{0.6cm}}_{\alpha\hspace{0.6cm}} \ar@{->>}[rd]_{\pi } &
 \PO /\PO (\mathfrak{n}') \ar@{->>}[r]  \ar@{->>}[d] &
\PO /\PO (\operatorname{rad}(\mathfrak{n}')) \ar@{-->>}[ld]\\
& \PSL (2,q)
}
\end{equation}
Now the rightmost projects onto
\begin{equation}\label{ProductOfPOones}
\PO /\PO (\mathfrak{p}_1)\times\cdots\times\PO /\PO (\mathfrak{p}_n),
\end{equation}
and by Corollary \ref{ChineseRemainderForPSL} the kernel of this projection is an abelian normal subgroup which is therefore mapped to the identity element by the dashed arrow in (\ref{DiagramWithDashedEpi}). Hence that dashed arrow is defined on (\ref{ProductOfPOones}); by Lemma \ref{LemmaOnFiniteGroups} it actually has to factor through the projection onto one of them, composed with an isomorphism. We hence obtain
\begin{equation*}
\xymatrix{
\Gamma /\Gamma (\mathfrak{n}') \ar[r]^{\simeq\hspace{0.6cm}}_{\alpha\hspace{0.6cm}} \ar@{->>}[rd]_{\pi } &
 \PO /\PO (\mathfrak{n}') \ar@{->>}[r]  \ar@{->>}[d] &
\PO /\PO (\mathfrak{p}_j) \ar@{-->}[ld]^{\simeq }\\
& \PSL (2,q)
}
\end{equation*}
for some $1\le j\le n$. We may shorten this to
\begin{equation}\label{DiagramWithDashedIso}
\xymatrix{
\Gamma /\Gamma (\mathfrak{p}_j) \ar[rr]^{\simeq\hspace{0.5cm}}_{\alpha '\hspace{0.5cm}} \ar@{->>}[rd]_{\pi '} & &
\PO /\PO (\mathfrak{p}_j) \ar@{-->}[ld]^{\simeq }\\
& \PSL (2,q)
}
\end{equation}
with $\alpha '$ again induced by the inclusion $\Gamma\subseteq\PO$. In this diagram $\pi '$ is obviously an isomorphism, therefore $\Delta =\ker\pi$ is equal to $\Gamma (\mathfrak{p}_j)$. The dashed isomorphism in (\ref{DiagramWithDashedIso}) shows that the norm of $\mathfrak{p}_j$ is $q$.
\end{proof}

\begin{Remark}\label{ReconstructionOfArithmeticSimilarityClass}
We note that this proposition enables us to reconstruct the splitting behaviour of almost all primes in $k$ from $\Gamma$ and its congruence subgroups: Let $p\notin S(\Gamma )$ be a rational prime in $\Gamma$. Then there exist only finitely many normal congruence subgroups $\Delta\triangleleft\Gamma$ such that $\Gamma /\Delta\simeq\PSL (2,q)$ for some power $q$ of $f$. Let these be $\Delta_1,\ldots ,\Delta_n$, and let the corresponding quotients be $\PSL (2,p^{f_1}),\ldots ,\PSL (2,p^{f_n})$.

On the other hand consider the prime decomposition $(p)=\mathfrak{p}_1\cdots\mathfrak{p}_m$ in $k$. Then $n=m$, and up to renumeration $\Delta_j=\Gamma (\mathfrak{p}_j)$ and $N(\mathfrak{p}_j)=p^{f_j}$. In particular we can reconstruct $[k\colon \bbq ]=f_1+\ldots +f_n$ from the knowledge of $\Gamma$ and its congruence subgroups.
\end{Remark}

\begin{proof}[Proof of Theorem A] By Theorem \ref{CullerShalenForPSLTwoR} we may replace $\Gamma_j$ by finite index subgroups corresponding to each other under the isomorphism $f$. Hence we may assume that each $\Gamma_j$ is torsion-free and satisfies the trace field condition. Again by Theorem \ref{CullerShalenForPSLTwoR} it suffices to show that $\tr^2 f(\gamma )=\tr^2 \gamma\in\bbr$ for each $\gamma\in\Gamma_1$.

Denote the trace field of $\Gamma_j$ by $k_j$. Each number $a\in\mathfrak{o}_{k_j}$ has a \emph{characteristic polynomial} $\chi_a(x)\in\bbz [x]$ which can be described as follows:
\begin{itemize}
\item it is the characteristic polynomial of the map $k_j\to k_j$, $v\mapsto av$ interpreted as a $\bbq$-linear map;
\item it is equal to $\prod_{\sigma }(x-\sigma (a))$. Here $\sigma$ runs through a system of representatives of $\Gal (L_j|\bbq )$ modulo $\Gal (L_j|k_j)$ where $L_j$ is the Galois closure of $k_j$.
\end{itemize}
Now let $p$ be a rational prime not in $S(\Gamma_1)\cup S(\Gamma_2)$. By Remark \ref{ReconstructionOfArithmeticSimilarityClass} we can decompose $p\mathfrak{o}_{k_j}$ into prime ideals
$$p\mathfrak{o}_{k_1}=\mathfrak{p}_1\cdots\mathfrak{p}_n,\quad p\mathfrak{o}_{k_2}=\mathfrak{q}_1\ldots\mathfrak{q}_n$$
in such a way that
\begin{equation}\label{HowfPermutesCongruenceSubgroups}
f(\Gamma_1(\mathfrak{p}_j))=\Gamma_2(\mathfrak{q}_j)\text{ and }\mathfrak{o}_{k_1}/\mathfrak{p}_j\simeq\mathfrak{o}_{k_2}/\mathfrak{q}_j.
\end{equation}
 Then
\begin{equation}\label{ProductsPjQj}
\mathfrak{o}_{k_1}/p\mathfrak{o}_{k_1}\simeq\mathfrak{o}_{k_1}/\mathfrak{p}_1\times\cdots\times\mathfrak{o}_{k_1}/\mathfrak{p}_d
\end{equation}
is a finite-dimensional $\bbf_p$-algebra, and we may similarly define the characteristic polynomial $\chi_{\overline{b}}(x)\in\bbf_p[x]$ of an element $\overline{b}\in\mathfrak{o}_{k_1}/p\mathfrak{o}_{k_1}$ as the characteristic polynomial of the $\bbf_p$-linear endomorphism of $\mathfrak{o}_{k_1}/p\mathfrak{o}_{k_1}$ given by multiplication by $\overline{b}$. Then for $a\in\mathfrak{o}_{k_1}$ clearly
\begin{equation}\label{ComputeCharPolyModp}
\chi_a(x)\bmod p =\chi_{a\bmod p}(x)\in\bbf_p[x].
\end{equation}
We now claim that the characteristic polynomials of $\tr^2\gamma$ and $\tr^2f(\gamma )$ are congruent modulo $p$. To see this we use the abstract version of squared traces on finite groups introduced in section \ref{SectionFiniteGroups}. For each $1\le j\le n$, using (\ref{HowfPermutesCongruenceSubgroups}) we obtain an isomorphism of finite groups $\bar{f}\colon\Gamma_1/\Gamma_1(\mathfrak{p}_j)\to\Gamma_2/\Gamma_2(\mathfrak{q}_j)$. By the remark after Definition \ref{AbstractTraces}, $\tr^2\gamma\bmod\mathfrak{p}_j$ and $\tr^2f(\gamma )\bmod\mathfrak{q}_j$ are Galois-conjugate elements of the finite field $\bbf_q\simeq\mathfrak{o}_{k_1}/\mathfrak{p}_j\simeq\mathfrak{o}_{k_2}/\mathfrak{q}_j$. Hence there exists an isomorphism of $\bbf_p$-algebras
$$\alpha_j\colon\mathfrak{o}_{k_1}/\mathfrak{p}_j\overset{\simeq}{\longrightarrow}\mathfrak{o}_{k_2}/\mathfrak{q}_j$$
with $\alpha_j(\tr^2\gamma \bmod\mathfrak{p}_j)=\tr^2f(\gamma )\bmod\mathfrak{q}_j$. Gluing these together component-wise in (\ref{ProductsPjQj}) yields an isomorphism of $\bbf_p$-algebras $\alpha\colon\mathfrak{o}_{k_1}/p\mathfrak{o}_{k_1}\to\mathfrak{o}_{k_2}/p\mathfrak{o}_{k_2}$ with $\alpha (\tr^2\gamma\bmod p)=\tr^2f(\gamma )\bmod p$. Since characteristic polynomials are stable under algebra isomorphisms, we obtain
$$\chi_{\tr^2\gamma\bmod p}(x)=\chi_{\tr^2f(\gamma )\bmod p}(x)\in\bbf_p[x].$$
By (\ref{ComputeCharPolyModp}), this means
$$\chi_{\tr^2\gamma }(x)\equiv\chi_{\tr^2f(\gamma )}(x)\bmod p.$$
But this holds for infinitely many $p$, so
$$\chi_{\tr^2\gamma }(x)=\chi_{\tr^2f(\gamma )}(x)\in\bbz [x].$$

Since we had assumed $\Gamma_1$ to be torsion-free, $\gamma$ cannot be elliptic. If it is parabolic, then $\tr^2\gamma =4$ and therefore $\chi_{\tr^2\gamma }(x)=(x-4)^d$. Hence also the characteristic polynomial of $f(\gamma )$ is $(x-4)^d$, and since $\tr^2f(\gamma )$ is a zero of this polynomial, $\tr^2f(\gamma )=4$, hence $f(\gamma )$ is parabolic as well.

Finally assume that $\gamma$ is hyperbolic. Then $f(\gamma )$ must also be hyperbolic because it cannot be parabolic (else $\gamma$ would be parabolic by the inverse of the previous argument). By Proposition \ref{ModularEmbeddingEnablesToPinDownTrace}, $\tr^2\gamma$ is the largest zero of $\chi_{\tr^2\gamma }(x)$, similarly for $\tr^2f(\gamma )$. Therefore $\tr^2(\gamma )=\tr^2f(\gamma )$.
\end{proof}

\section{An Example}

\noindent In our proof of Theorem A we did not use the full assumption that all congruence subgroups are mapped to congruence subgroups by the given isomorphism. We spell out in a concrete example how far an isomorphism between non-conjugate arithmetic groups can be from preserving congruence subgroups.

In \cite{MR702765} we find a complete list of all arithmetic groups of signature $(1;2)$, i.e.\ whose associated Riemann surfaces have genus one and which have one conjugacy class of elliptic elements, these elements being of order two. In particular all these groups are abstractly isomorphic, and we may just pick the first two of them:
$\Gamma_1'$ is generated by the two M\"{o}bius transformations
$$\alpha_1=\pm\begin{pmatrix}\frac{1+\sqrt{5}}{2} & 0\\ 0&\frac{-1+\sqrt{5}}{2}\end{pmatrix}\text{ and }\beta_1=\pm\begin{pmatrix}\sqrt{3} &\sqrt{2}\\ \sqrt{2} &\sqrt{3}\end{pmatrix}\! ,$$
$\Gamma_2'$ by the two M\"{o}bius transformations
$$\alpha_2=\pm\begin{pmatrix}\sqrt{2}+1 & 0\\ 0&\sqrt{2}-1\end{pmatrix}\text{ and }\beta_2=\pm\frac{1}{2}\begin{pmatrix}\sqrt{6} &\sqrt{2}\\ \sqrt{2} &\sqrt{6}\end{pmatrix}\! .$$
These are, respectively, generators satisfying the relation $(\alpha_j\beta_j\alpha_j^{-1}\beta_j^{-1})^2=1$. So there exists a group isomorphism $f\colon\Gamma_1'\to\Gamma_2'$ with $f(\alpha_1)=\alpha_2$ and $f(\beta_1)=\beta_2$. The $\Gamma_j'$ do not satisfy the trace field condition, but the $\Gamma_j=(\Gamma_j')^{(2)}$ (between whom $f$ also induces an isomorphism) do; in both cases the invariant trace field is $\bbq$.

Then, with finitely many exceptions, $\Gamma_1/\Gamma_1(p)\simeq\PSL (2,p)\simeq\Gamma_2/\Gamma_2(p)$ for rational primes $p$; nevertheless, the proof of Theorem A shows that there can be only finitely many $p$ such that $f(\Gamma_1(p))$ is a congruence subgroup (and hence only finitely many $p$ with $f(\Gamma_1(p))=\Gamma_2(p)$).

\section{Concluding Remarks}

\begin{Remark}
The assumption that both groups admit a modular embedding is crucial although it only enters in the very last step of the proof. If $\Gamma$ is a semi-arithmetic lattice with invariant trace field $k$ and $\sigma\colon k\to\bbr$ a field embedding we obtain in a natural way a group $i_\sigma (\Gamma )\subset\PSL (2,\bbr )$, see \cite[Remark 4]{MR1745404}. There exist semi-arithmetic lattices $\Gamma$ with nontrivial Galois conjugates $i_\sigma (\Gamma )$ that are again lattices, and then the isomorphism $\Gamma\to i_\sigma (\Gamma )$ preserves congruence subgroups but not traces. For an explicit construction see e.g.\ \cite{AgolNew} referring to \cite[Proposition 4.11]{MR1643429}. But if $\Gamma$ admits a modular embedding, then none of the nontrivial Galois conjugates $i_\sigma (\Gamma )$ can be discrete by \cite[Theorem 3]{MR1745404}.

Note that the existence of a modular embedding enters the proof via Proposition \ref{ModularEmbeddingEnablesToPinDownTrace} which is its only genuinely non-algebraic ingredient: it is a consequence of the Schwarz Lemma.

One may still ask whether a weakened version of our main theorem holds in the general case: if $f\colon\Gamma_1\to\Gamma_2$ is an isomorphism between semi-arithmetic lattices in $\PSL (2,\bbr )$ respecting congruence subgroups, is it the composition of an inner automorphism of $\PGL (2,\bbr )$ with a Galois conjugation of the trace field?
\end{Remark}

\begin{Remark}
There exist arithmetic Fuchsian groups with different trace fields but whose congruence completions are isomorphic away from a finite set of primes. To see this, start with the polynomial in the remark after \cite[Theorem 5.1]{MR1800034}: the splitting field of this polynomial is a totally real Galois extension of $\bbq$ with Galois group $\PSL (2,7)$. By the discussion in \cite[p. 358--359]{MR0447188} such a field contains two subfields $k_1$, $k_2$ which are not isomorphic but have the same Dedekind zeta function. Then there exists a finite set $S$ of rational primes such that $\bba_{k_1}^S\simeq\bba_{k_2}^S$. From this we can easily construct arithmetic Fuchsian groups over $k_1$ and $k_2$ with isomorphic prime-to-$S$ congruence completion.

There also exist non-isomorphic number fields with isomorphic finite ad\`{e}le rings (at all primes), see \cite{MR736723}. But no construction seems to be known where these fields are totally real.
\end{Remark}

\providecommand{\bysame}{\leavevmode\hbox to3em{\hrulefill}\thinspace}
\providecommand{\MR}{\relax\ifhmode\unskip\space\fi MR }
\providecommand{\MRhref}[2]{%
  \href{http://www.ams.org/mathscinet-getitem?mr=#1}{#2}
}
\providecommand{\href}[2]{#2}

\end{document}